\documentclass{compositio} % ’we are using the compositio class’
\usepackage{amsmath,amsfonts,amssymb,enumerate,amsthm}
\usepackage{comment}
\usepackage{graphicx}
\usepackage[all]{xy}
\usepackage[OT2, T1]{fontenc}

\newtheorem{theorem}{Theorem}[section]
\newtheorem{proposition}[theorem]{Proposition}
\newtheorem{lemma}[theorem]{Lemma}
\newtheorem{corollary}[theorem]{Corollary}
\newtheorem{definition}[theorem]{Definition}

\newtheorem{claim}[theorem]{Claim}
\theoremstyle{remark}
\newtheorem{remark}[theorem]{Remark}

\begin{document}
\title
[Albanese cokernel] % ’optional short form; for the running head’ 
{A higher dimensional generalization of Lichtenbaum duality in terms of the Albanese map} % ’Main title’
% ’Each author has his or her own set of coordinates.’ 
\author{Wataru Kai}
\email{kaiw@ms.u-tokyo.ac.jp} % ’optional’ 
\address{Graduate School of Mathematical Sciences\\ The University of Tokyo \\3-8-1 Komaba\\%
Meguro-ku, Tokyo\\153-8914 Japan}
%\curraddr{Mathematics Institute\\Another University\\
%There\\BA1 1HZ} % ’also optional’ % ’Another author’
\shortauthors{W.~Kai}
% ’optional short form; for the running head’
% ’MSC classification, keywords and grant acknowledgements’ 
\thanks{The author is supported by the Program for Leading Graduate 
Schools, MEXT, Japan, and by Japan Society for the Promotion of Science through JSPS KAKENHI Grant Number 15J02264.}
% ’Abstract comes before maketitle, as in the AMS classes’ 
\begin{abstract}
In this article, 
we present a conjectural formula describing the cokernel of the Albanese map 
%for
of zero-cycles 
%\[\mathrm{coker} \left( \mathrm{CH}_0(X)^{\deg =0} \to \mathrm{Alb}_X(K) \right) \]
of smooth projective varieties $X$ over $p$-adic fields in terms of the N\'eron-Severi group and provide a proof under additional assumptions 
on an integral model of $X$.
The proof depends on a non-degeneracy result of Brauer-Manin pairing due to Saito-Sato and on Gabber-de Jong's comparison result of cohomological- and Azumaya-Brauer groups.
We will also mention the local-global problem for the Albanese-cokernel; the 
abelian group on the ``local side'' turns out to be a finite group.
\end{abstract}
\maketitle
% ’Main text starts here.’

%\usepackage{fancyhdr}
%\pagestyle{myheadings}
%\markboth{Albanese cokernel}{Albanese cokernel}

%%%A cycle-theoretic generalization of Lichtenbaum's duality to higher dimensions
%%%The Albanese map and the Brauer group over $p$-adic fields
%%%A higher dimensional generalization of Lichtenbaum's duality theorem in terms of the Albanese map
%\title{A higher dimensional generalization of Lichtenbaum duality in terms of the Albanese map}
%\author{Wataru Kai}
%\thanks{Graduate School of Mathematical Sciences, the University of Tokyo.
%Supported by the Program for Leading Graduate 
%Schools, MEXT, Japan.
%\newline
%\textit{E-mail address}: kaiw@ms.u-tokyo.ac.jp
%\newline
%\textit{Date}: \today
%\newline
%2010 Mathematics Subject Classification: 14F22, 14G20}}

%\date{}

%\href{mailto:kaiw@ms.u-tokyo.ac.jp}{kaiw@ms.u-tokyo.ac.jp}
%%%
%%%
%%%

%\newcommand{\qed}{\hfill $\blacksquare$ \par \vspace{10pt}}%QED%

\maketitle

%\tableofcontents

%%%%%%%%%%%%%%%%%%%%%%%%%%%%%
\section{Introduction}%%%%%%%
%%%%%%%%%%%%%%%%%%%%%%%%%%%%%
Let $K$ be a $p$-adic field with residue field $k$
and $X$ be a smooth projective variety over $K$.
In this paper, the term ``variety'' always means a geometrically irreducible scheme
over a field.

When $X$ is a curve, Lichtenbaum \cite{lich} defined and studied the pairing
\begin{equation}
\mathrm{Pic}(X) \times \mathrm{Br}(X) \to \mathbb{Q/Z} \tag{\text{L}}
\end{equation}
which he shows to be non-degenerate, where $\mathrm{Br}(X)$ denotes the \'etale cohomology
grouop $H^2_{\mathrm{et}}(X,\mathbb{G}_m)$. Further, with the aid of Tate duality for Abelian varieties over
$p$-adic fields, he proved that the cokernel of the map
\[
\mathrm{Pic}^0(X) \to J_X(K)
\]
is canonically Pontryagin dual to the cokernel of the degree map
$\mathrm{Pic}(\overline{X})^{G_K} \to \mathbb{Z}$.
Here $J_X$ denotes the Jacobian variety of $X$ and $\overline{X}$ is the base change to an algebraic closure $\overline{K}$ of $K$. ${G_K}$ denotes the absolute Galois group of $K$.

\vspace{10pt}
In this paper we present an attempt to generalize Lichtenbaum's result to higher dimensional $X$.
Namely, we study the cokernel of the Albanese map
\[
\mathrm{alb}_X: \mathrm{CH}_0(X)^0\to \mathrm{Alb}_X(K)
\]
where $\mathrm{CH}_0(X)^0$ denotes the degree $0$ part of the Chow group $\mathrm{CH}_0(X)$.

It is easy to see that this cokernel is torsion, is a birational invariant 
(i.e.~birational morphisms induce isomorphisms)
of smooth 
proper varieties over any field.
Over $p$-adic fields, it is finite by the argument of Saito and Sujatha \cite{sasu}.
Except these basic facts, little was known. 

Our main result is the following.
%%%
%%%%%%%%%%%%%%%%%%%%%%%%%%
\begin{theorem}[(\textit{see} Theorem \ref{th1} and Remark \ref{cor})]
Let $X$ be a smooth projective variety over a $p$-adic field $K$ and suppose $X$ admits a smooth projective
model $\mathcal{X}$ over the integer ring $\mathcal{O}_K$
whose Picard scheme $\mathcal{\bf Pic}_{\mathcal{X}/\mathcal{O}_K}$ over $\mathrm{Spec}(\mathcal{O}_K)$ (which exists in this case) is smooth.

Then the cokernel of the map
\[
\mathrm{alb}_X: \mathrm{CH}_0(X)^0 \to \mathrm{Alb}_X(K)
\]
is canonically Pontryagin dual to the cokernel of
\[
\mathrm{Pic}(\overline{X})^{G_K} \to \mathrm{NS}(\overline{X})^{G_K}.
\]
Moreover, the last cokernel is a subquotient of the $p$-primary torsion subgroup of
$\mathrm{NS}(\overline{X})^{G_K}$.
\end{theorem}
%%%%%%%%%%%%%%%%%%%%%%%%%%
%%%
Here $\mathrm{NS}$ denotes the N\'eron-Severi group: for any variety $Y$ over a field,
$\mathrm{NS}(Y)$ is defined to be the quotient $\mathrm{Pic}(Y)/\mathrm{Pic}^0(Y)$.
Smoothness condition on the Picard scheme is satisfied if
$H^2(X,\mathcal{O}_X)$ and $H^2(Z,\mathcal{O}_Z)$ vanish
(here we set $Z:=\mathcal{X}\otimes _{\mathcal{O}_K} k$).

In the proof of the theorem, we interpret the pairing $(\text{L})$ as 
the Brauer-Manin pairing introduced in \cite{manin}
\begin{equation}
\mathrm{CH}_0(X) \times \mathrm{Br}(X) \to \mathrm{Br}(K) \tag{\text{BM}} .
\end{equation}
This pairing has been used to study zero-cycles on varieties, especially over $p$-adic fields
and number fields.
By non-degeneracy results on the Brauer-Manin pairing due to S. Saito and K. Sato,
together with the Tate pairing, 
we reduce our problem to an injectivity problem concerning Brauer groups.
We solve it using arguments of M. Artin and an existence theorem of Azumaya algebras due to Gabber and de Jong.

%%%
%%%
%%%
\vspace{10pt}
There is another surjectivity result, proved separately:
%%%
%%%%%%%%%%%%%%%
\begin{theorem}[(\textit{see} Theorem \ref{th2})]
Let $X$ be a smooth projective variety over $K$ with good reduction.
Suppose the ramification index of $K$ over $\mathbb{Q}_p$ is $<p-1$.
Then the Albanese map for $X$
\[
\mathrm{alb}_X : \mathrm{CH}_0(X)^0\to \mathrm{Alb}_X(K)
\]
is surjective.
\end{theorem}
%%%%%%%%%%%%%%%
%%%
The proof of this theorem is based on 
knowledge about the behavior of N\'{e}ron models
of abelian varieties
under reduction.

%%%
%%%
%%%
\vspace{10pt}
Note that given a smooth projective variety $X$ over a \textbf{number field} $K$
(here ``projective'' can be replaced by ``proper''; \S 2),
we can use Theorem 1.2 to show that the map
\[
\mathrm{CH}_0(X_{K_v})^0 \to \mathrm{Alb}_X(K_v)
\]
is surjective for almost all places $v$ of $K$.
Here $K_v$ denotes the completion of $K$ at $v$.
So we are tempted to consider the map between finite groups
\[
\mathrm{Alb}_X(K)/ \mathrm{alb}_X\mathrm{CH}_0(X)^0 \to
\prod _{v:\text{ all places}}\mathrm{Alb}_X(K_v)/\mathrm{alb}_{X_{K_v}}\mathrm{CH}_0(X_{K_v})^0 .
\]
When $X$ is a curve, its injectivity follows from the Hasse principle for the Brauer group.
%It would be interesting to ask if it is injective in higher dimensional cases.
An interesting question is whether or not it is injective in the higher dimensional case as well.

\vspace{10pt}
Another generalization of Lichtenbaum's duality concerning the Picard variety
and the Picard group, which is dual to our point of view, was considered comprehensively by van Hamel \cite{hamel}.

\section{Preliminaries}

\subsection{Albanese cokernel: definition and basic 
properties}

We include the definition of the Albanese map
for completeness.

Let $Y$ be a smooth proper variety over a field $F$.
Then there exist an abelian variety $\mathrm{Alb}_Y$ over $F$
and a morphism $\phi : Y\times _F Y\to \mathrm{Alb}_Y$
which satisfy the following property:
$\phi |_{\Delta _Y}$ is the constant map to zero,
where $\Delta _Y \subset Y\times Y$ is the diagonal subscheme;
given a morphism $\psi :Y\times Y\to A$ into an abelian
variety $A$ which is the constant map to zero 
on the diagonal, there is a unique homomorphism of
abelian varieties $f:\mathrm{Alb}_Y\to A$ with
$\psi =f\circ \phi$ \cite[pp.45--46]{lang}.

In this case, $Y$ being proper, 
the Albanese variety $\mathrm{Alb}_Y$ of $Y$ is 
also characterized by the following property
(cf.~\cite[Lemma 2.3]{spacefilling}):
every time we choose an extension field $L$ of $F$ 
and an $L$-valued point $x_0$ of $Y$, 
$(\mathrm{Alb}_Y)_L$ co-represents the functor
\[ \begin{array}{ccccccccc}
\{ \text{Abelian varieties over $L$} \} &\to
 &\{ \text{Sets} \} \\%
A &\mapsto &\mathrm{Hom}_{*}((Y_L,x_0),(A,0)),
\end{array}\]
where $\mathrm{Hom}_*$ denotes the set of 
basepoint-preserving morphism of $L$-schemes. 
In particular there is a universal morphism
$\phi _{x_0}:Y_L\to (\mathrm{Alb}_Y)_L$ sending $x_0$ to $0$.

By Galois descent, there is a torsor $\mathrm{Alb}_Y^1/F$
under $\mathrm{Alb}_Y$ and a morphism
$\phi ':Y\to \mathrm{Alb}_Y^1$ which has a universal 
property:
given a morphism $\psi ':Y\to A'$ into a torsor $A'$
under an abelian variety $A$,
there is a unique homomorphism $f:\mathrm{Alb}_Y
\to A$ and a unique morphism $f':\mathrm{Alb}_Y^1\to A'$
which are compatible.

The torsor $\mathrm{Alb}_Y^1$ is determined by 
an element $[\mathrm{Alb}_Y^1]\in 
H^1_{\mathrm{et}}(F,\mathrm{Alb}_Y)$.
By the canonical isomorphism
$H^1_{\mathrm{et}}(F,\mathrm{Alb}_Y)\cong
\mathrm{Ext}^1_{(\mathrm{Sch}/F)_{\mathrm{et}}}
(\mathbb{Z},\mathrm{Alb}_Y)$,
it corresponds to an extension of group schemes
over $F$
%%%%%%%%%%%%%%%%
\begin{equation}\label{albscheme}  
0\to \mathrm{Alb}_Y\to A_Y\xrightarrow{\deg} \mathbb{Z}\to 0 
\end{equation}
%%%%%%%%%%%%%%%%
where $A_Y$
is described explicitly as follows.
Denote by $\mathrm{Alb}_Y^i$ 
the $i$-th power of the torsor $\mathrm{Alb}_Y^1$,
which corresponds to $i\cdot [\mathrm{Alb}_Y^1]
\in H^1_{\mathrm{et}}(F,\mathrm{Alb}_Y)$.
Set $A_Y:= \coprod _{i\in \mathbb{Z}} \mathrm{Alb}_Y^i$.
The morphisms $\mathrm{Alb}_Y^i \times _F \mathrm{Alb}_Y^j
\to (\mathrm{Alb}_Y^i \times _F \mathrm{Alb}_Y^j)/\mathrm{Alb}_Y
=\mathrm{Alb}_Y^{i+j}$
and $\mathrm{Alb}_Y^{-i}\cong (\mathrm{Alb}_Y^i)^{-1}$ 
make $A_Y$ into a group scheme.
Define a map $\mathrm{Alb}_Y\to A_Y$ as the canonical
open and closed immersion to the $0$th component,
and $A_Y\to \mathbb{Z}$ to be the constant map to $i$
on $\mathrm{Alb}_Y^i$.
Being a group scheme with a quasi-projective 
neutral component, 
$A_Y$ has a transfer structure
\cite[Proof of Lemma 3.2]{spsz}, \cite[Lemma 1.3.2]{BV-K}
with respect to which (\ref{albscheme}) is an exact sequence
of \'etale sheaves with transfers.

We define the map $\mathrm{alb}_Y:\mathrm{CH}_0(Y)^0\to \mathrm{Alb}_Y(F)$.
First we define a map 
$\mathrm{alb}_Y:Z_0(Y)\to A_Y(F)$.
This map is determined by determining the image of each
$[x]$, $x\in Y_{(0)}$.
Let $F(x)$ be the residue field of $x$.
There is a canonical element $x\in Y(F(x))$.
Let $\mathrm{alb}_Y([x])$ be the image of $x$ by
\[ Y(F(x))\xrightarrow{\phi '} \mathrm{Alb}_Y^1(F(x))\subset A_Y(F(x))
\xrightarrow{\text{transfer}} A_Y(F).  \]
By the fact that maps in (\ref{albscheme}) are compatible
with transfer structures the following diagram commutes:
\[  \begin{array}{ccccccccc}
Z_0(Y)&\xrightarrow{\deg }&\mathbb{Z} \\%
\downarrow _{\mathrm{alb}_Y}&&\Vert \\%
A_Y(F)&\xrightarrow{\deg}&\mathbb{Z}.
\end{array}\]
Therefore a map 
\[\mathrm{alb}_Y:Z_0(Y)^0\to \mathrm{Alb}_Y(F)\]
between degree $0$ parts is induced.
%This is the definition of the Albanese map.
It is known that it factors through
$\mathrm{CH}_0(Y)^0$. 
%by reduction to the case $Y$ is a curve.

%%%%%%%%%%%%%%%%%%%%%%%%
\begin{definition}
The cokernel of the map
\[ \mathrm{alb}_Y: \mathrm{CH}_0(Y)^0\to \mathrm{Alb}_Y(F)   \]
is called the {\bf Albanese cokernel} of $Y$.
\end{definition}
%%%%%%%%%%%%%%%%%%%%%%%

The following properties are easy to verify.
%%%%%%%%%%%%%%%%%%%%%%
\begin{proposition}
The Albanese cokernel of proper smooth varieties satisfies the following:
\begin{enumerate}[(i)]
\item It is trivial if $F$ is an algebraically closed field
or if $F$ a finite field and $Y$ is projective.
\item It is trivial if $Y$ is an abelian variety.
\item Albanese cokernels form a covariant functor from the category of
smooth proper varieties to the category of abelian groups.
\item Albanese cokernels are contravariant with respect to 
base-changes $\mathrm{Spec}F'\to \mathrm{Spec}F$ 
($F'$ is a field) and covariant if $F'/F$ is finite.
Consequently, they are torsion.
\item The Albanese cokernel is finite if $F$ is $\mathbb{R}$, a $p$-adic
field or is a finitely generated field over a prime field.
\end{enumerate}
\end{proposition}
%%%%%%%%%%%%%%%%%%%%%%
\begin{proof}
The assertion (i) in the case $F$ is a finite field
follows from Kato-Saito's class field theory \cite[Proposition 9(1)]{KS} where they further describe the kernel of the Albanese map.
The assertion (v) in the case $F$ is a finitely generated
field follows from Mordell-Weil theorem
together with the fact that it is torsion.
The case $F$ is a $p$-adic field follows from arguments in
\cite[p.409]{sasu}. The case $F=\mathbb{R}$ follows from knowledge on the structure of commutative Lie groups.
\end{proof}

%%%%%%%%%%%%%%%%%%
\begin{proposition}[(Rigidity)]\label{comphens}
Let $K$ be a Henselian non-archimedean $\mathbb{R}$-valued 
valuation field of characteristic $0$ 
and $\hat{K}$ be its completion. Let $X$ be a smooth proper variety.
Then, the map
\[ \frac{\mathrm{Alb}_X(K)}{\mathrm{alb}_X(\mathrm{CH}_0(X)^0)} 
\to \frac{\mathrm{Alb}_X(\hat{K})}
{\mathrm{alb}_{X_{\hat{K}}}(\mathrm{CH}_0(X_{\hat{K}}))^0 } \]
is an isomorphism.
\end{proposition}
%%%%%%%%%%%%%%%%%%
\begin{proof}
We define a homotopy invariant presheaf 
with transfers $F$ on $\mathrm{Sm}/K$
(the category of separated smooth schemes 
of finite type) by
\[
Y \mapsto \mathrm{Alb}_X(Y)/\mathrm{Im}~h_0(X/K)(Y)^0
\]
where $h_0(X/K)$ denotes the functor
\[  \begin{array}{ccccc}
\mathrm{Sm}/K&\to &(\text{Abelian groups}) \\
Y &\mapsto & \mathrm{Cok}\left( \mathbb{Z}_{\mathrm{tr}}(X)(Y\times \mathbb{A}^1)\xrightarrow{d}\mathbb{Z}_{\mathrm{tr}}(X)(Y) \right)
\end{array}\]
($d$ is the difference of the evaluation maps at $0$ and $1\in \mathbb{A}^1$),
and
\[ h_0(X/K)(Y)^0 := \mathrm{ker}\left( h_0(X/K)(Y)
\xrightarrow{\deg } \mathbb{Z}^{\pi _0(Y)}\right) \]
(see \cite{spsz} for the morphism $h_0(X/K)^0
\to \mathrm{Alb}_X$). 

Since $F$ takes torsion values on the spectra of 
fields, its Zariski sheafification
is a torsion abelian presheaf with transfers which 
is homotopy invariant
\cite[Corollary 4.19 and Proposition 4.26]{preth}.
Now the assertion follows by applying the rigidity 
theorem with respect to the extension
$K\subset \hat{K}$
\cite[Theorem 1]{RO} to the sheaf $F_{\mathrm{Zar}}$.
\end{proof}

%PPPPPPPPPPPPPPPPPPPPPPP
\begin{proposition}[(Birational invariance)]
Let $Y'\to Y$ be a birational morhpism of smooth 
proper
varieties over $F$. Then it induces an isomorphism
on their Albanese cokernel.
\end{proposition}
%PPPPPPPPPPPPPPPPPPPPPP
\begin{proof}
We know that the Albanese variety is a birational 
invariant
and the map $\mathrm{CH}_0(Y')^0\to \mathrm{CH}_0(Y)^0$ is surjective. 
Therefore the induced map on
%$\mathrm{Alb}_{(-)}(K)/\mathrm{alb}(\mathrm{CH}_0(-)^0)$
the Albanese cokernel is 
a bijection.
\end{proof}

%%%
%%%%%PPPPPPPPPPPPPPPPPPPPPPPPPPPPPPPPP
\begin{proposition}[(Hypersurface sections)]
Let $K$ be a Henselian discrete valuation field 
of characteristic $0$
with finite residue field. Let $X\subset \mathbb{P}^n$ 
be a smooth projective 
variety over $K$ of dimension $\ge 3$.
Then there is a smooth hypersurface section $H$ 
of $X$ for which the map
\[
\frac{\mathrm{Alb}_H(K)}{\mathrm{alb}_H\mathrm{CH}_0(H)^0}
\to 
\frac{\mathrm{Alb}_X(K)}{\mathrm{alb}_X\mathrm{CH}_0(X)^0}
\]
is an isomorphism.
\end{proposition}
%%%%%
%%%

\begin{proof} 
First suppose $K=\hat{K}$ is a 
$p$-adic field.
Then $\mathrm{alb}_X\mathrm{CH}_0(X)^0\subset \mathrm{Alb}_X(\hat{K})$
 contains an open subgroup $U$
of $\mathrm{Alb}_X(\hat{K})$ isomorphic to a direct sum 
of finitely many copies of $\mathcal{O}_{\hat{K}}$
(cf. \cite[Theorem 7]{mat}, \cite[p.409]{sasu}).
Therefore $\mathrm{alb}_X\mathrm{CH}_0(X)^0$ is topologically 
generated by finitely many zero-cycles
$a_1, \ldots , a_n$.
We choose finitely many closed points
$x_1, \ldots , x_m \in X$ such that 
%the images of 
%zero cycles with support on them 
%generate $\mathrm{alb}_X\mathrm{CH}_0(X)^0$ topologically,
$a_1,\dots ,a_n$ have supports on $\{ x_1,\dots ,x_m\} $
and choose a smooth hypersurface section $H$ passing through 
these points
(possible by \cite[Theorem (7)]{AK}).
As $\mathrm{dim}(X)\ge 3$, $\mathrm{Alb}_H\to 
\mathrm{Alb}_X$ is an isomorphism.
\[\begin{array}{ccc}
\mathrm{CH}_0(H)^0 &\to &\mathrm{Alb}_H({\hat{K}})\\
\downarrow &&\downarrow \cong\\
\mathrm{CH}_0(X)^0 &\to &\mathrm{Alb}_X({\hat{K}})
\end{array}\]
Since $\mathrm{alb}_H\mathrm{CH}_0(H)^0\subset \mathrm{Alb}_X({\hat{K}})$ 
is an open subgroup of $\mathrm{alb}_X\mathrm{CH}_0(X)^0$
containing $a_1, \ldots , a_n$, it coincides with 
$\mathrm{alb}_X\mathrm{CH}_0(X)^0$. 

Next, we consider the general case. Let $\hat{K}$ be the completion of $K$. Over $\hat{K}$, we have proved the existence of a hypersurface section $H$ with the desired property. By Proposition \ref{comphens}, it suffices to show that such an $H$ can be taken over $K$. For that it is sufficient to find zero-cycles $a_1,\dots ,a_n$ as above which are defined over $K$.
So we are going to modify the cycles $a_i$ obtained above to find zero-cycles $a_i'$ of degree $0$ defined over $K$ such that the elements $\mathrm{alb}_{X}(a_i')\in \mathrm{Alb}_X(K)\subset \mathrm{Alb}_X(\hat{K})$ generate the subgroup $\mathrm{alb}_{X_{\hat{K}}}(\mathrm{CH}_0(X_{\hat{X}})^0)$ topologically.

Observe that if we change each of the elements $\mathrm{alb}(a_i)\in \mathrm{Alb}_X(\hat{K})$ by an element of $pU$ (here $U$ is the open subgroup of $\mathrm{alb}_{X_{\hat{K}}}(\mathrm{CH}_0(X_{\hat{K}})^0)\subset \mathrm{Alb}_X(\hat{K})$ mentioned earlier in this proof), the property that they are a set of topological generators of $\mathrm{alb}(\mathrm{CH}_0(X_{\hat{K}})^0)$ does not change.

Now, for the time being, consider an arbitrary zero-cycle
\[ {a}=\sum _\alpha [x_\alpha ] - \sum _{\beta }[x_{\beta }] \]
on $X$ (or on $X_{\hat{K}}$) of degree $0$. They define a point
\[ \mathbf{x}({a})=((x_\alpha )_{\alpha },(x_{\beta })_{\beta }) \in \prod _{\alpha }X(K(x_{\alpha })) \times \prod _{\beta }X(K(x_{\beta })).  \]
We have the albanese map
%%% Eq %%%
\[ \begin{array}{rlccc} \prod _{\alpha }X(K(x_{\alpha })) \times \prod _{\beta }X(K(x_{\beta }))
\longrightarrow
& \prod _{\alpha } \mathrm{Alb}_X^1(K(x_{\alpha }))\times \prod _{\beta } \mathrm{Alb}_X^1(K(x_{\beta })) \\
 \xrightarrow[\sum _{\alpha }(\text{transfer})-\sum _{\beta }(\text{transfer})]{}
& \mathrm{Alb}_X(K)
\end{array} \]
%%% Eq ends %%%
where the addition and subtraction take place in $A_X(K)$, and the map goes into the neutral component $\mathrm{Alb}_X$ because of the degree $0$ assumption. This map sends $\mathbf{x}(a)$ to $\mathrm{alb}_X({a})$. The map is continuous for the valuation topology on both sides because the transfer map is based on scheme morphisms \cite[Lemma 3.2]{spsz}.

Getting back to our situation, let us write 
\[ a_i=\sum _{\alpha \in A_i}[x_\alpha ]-\sum _{\beta \in B_i}[x_{\beta }] .  \]
Then $\mathbf{x}(a_i)$ is a point in $\prod _{\alpha \in A_i}X(\hat{K}(x_{\alpha }))\times \prod _{\beta \in B_i}X(\hat{K}(x_{\beta }))$.
The observation made above and the continuity implies that we may replace each of our $a_i$ (a zero-cycle on $X_{\hat{K}}$) by a zero-cycle represented by a point sufficiently close to $\mathbf{x}(a_i)$. 

Since the extension $\hat{K}(x_\alpha )/\hat{K}$ is separable, the integral closure $K_\alpha $ of $K$ in $\hat{K}(x_{\alpha })$ is dense in $\hat{K}(x_{\alpha })$ so that we have $(K_{\alpha })\hat{~}=\hat{K}(x_{\alpha })$ (and similarly for $\beta $). Then by an approximation theorem \cite[\S 3.6, Corollary 10]{BLR}, the subset 
\[ X(K_{\alpha })\subset X(\hat{K}(x_{\alpha })) \] 
is dense (and similarly for $\beta $). So we can take a point in $\prod _{\alpha \in A_i}X({K}_{\alpha })\times \prod _{\beta \in B_i}X({K}_{\beta })$ sufficiently close to $\mathbf{x}(a_i)$, and we let $a'_i$ be the zero-cycle on $X$ represented by the point. Then the elements 
\[ \mathrm{alb}_X(a'_i)\in \mathrm{Alb}_X(\hat{K}) \]
generate the subgroup $\mathrm{alb}_{X_{\hat{K}}}(\mathrm{CH}_0(X_{\hat{K}})^0)$ topologically.
This completes the proof.
\end{proof}

\subsection{Automorphisms and Brauer groups}\label{sec_auto}

\subsubsection{}\label{auto_gal}
Suppose $f:Y\to X$ is a Galois covering of schemes.
We have the Hochschild-Serre spectral sequence
\[ E^{i,j}_2= H^i(\mathrm{Gal}(Y/X),H^j_\mathrm{et} (Y,\mathbb{G}_m)) 
\Rightarrow H^{i+j}_\mathrm{et}(X,\mathbb{G}_m), \]
from which we get a homomorphism (where we write 
$\mathrm{Br}(Y/X):=\ker \left( \mathrm{Br}(X)\to \mathrm{Br}(Y)\right)$)
\[ E^2_1 = \mathrm{Br}(Y/X) \to E^{1,1}_2=H^1(\mathrm{Gal}(Y/X),\mathrm{Pic}(Y)) .  \]
We will denote it by $\phi _{Y/X}$.

\subsubsection{}
We recall an Azumaya algebra on a scheme $Y$ is by definition
a sheaf $\mathcal{A}$ of $\mathcal{O}_Y$-algebras
(not necessarily commutative) which is a locally free $\mathcal{O}_Y$-module of finite rank
and such that the canonical morphism
\[
\mathcal{A}\otimes _{\mathcal{O}_Y} \mathcal{A}^{\mathrm{op}}\to
\mathcal{E}\mathit{nd}_{\mathcal{O}_Y\text{-mod.}}(\mathcal{A})
\]
is an isomorphism. Two Azumaya algebras $\mathcal{A}$ and $\mathcal{A}'$ are said to be
equivalent if there are vector bundles $\mathcal{V}$ and $\mathcal{V}'$ on $Y$
and an isomorphism of algebras $\mathcal{A}\otimes _{\mathcal{O}_Y}
\mathcal{E}{nd}(\mathcal{V})
\cong \mathcal{A}'\otimes _{\mathcal{O}_Y}\mathcal{E}\mathrm{nd}(\mathcal{V}')$.
The set of equivalence classes of Azumaya algebras form a torsion abelian group
by tensor products. 
Let us denote this group by $\mathrm{Br_{Az}}(Y)$.
There is a canonical injection
$\mathrm{Br_{Az}}(Y) \hookrightarrow \mathrm{Br}(Y)_{\mathrm{tors}}$
where $(-)_{\mathrm{tors}}$ indicates the torsion subgroup.
This is not surjective in general.
Nevertheless we have
%%%
%%%%%%%%%%%%%%%%%%%%%%%
\begin{theorem}[({de Jong \cite[Theorem 1.1]{dejong}})]\label{gabthm}
Let $X$ be a quasi-compact separated scheme which admits an ample line bundle.
Then the map $\mathrm{Br_{Az}}(X) \to \mathrm{Br}(X)_{\mathrm{tors}}$ is bijective.
\end{theorem}
%%%%%%%%%%%%%%%%%%%%%%%
%%%
Note that if moreover $X$ is regular and noetherian,
we have $\mathrm{Br_{Az}}(X) = \mathrm{Br}(X)$
since the Brauer group is known to be torsion
for regular noetherian schemes.

\subsubsection{}
Suppose $f:Y\to X$ is a morphism of schemes and an abstract group
$G$ acts on $Y$ (on the right) over $X$.
For $\sigma \in G$, denote by $[\sigma ]:Y\to Y$ the corresponding action;
we have $[\sigma ][\tau ]= [\tau \sigma ]$.
Then $G$ acts on $\mathrm{Pic}(Y)$ on the left by 
$\mathcal{L} \mapsto [\sigma ] ^* \mathcal{L}$ for $\sigma \in G$.

In this case we have a homomorphism
\[ \phi _{G, Y/X}: \mathrm{Br}_{\text{Az}}(Y/X)\to H^1(G,\mathrm{Pic}(Y)) \]
(where $\mathrm{Br}_{\text{Az}}(Y/X)
:=\ker \left(
\mathrm{Br}_{\text{Az}}(X)\to \mathrm{Br}_{\text{Az}}(Y)
\right)$) described as follows.

Let $\omega \in \mathrm{Br}_{Az}(Y/X)$. Take an Azumaya algebra $\mathcal{A}$ on $X$
which represents $\omega$.
There is a vector bundle $\mathcal{E}$ on $Y$ and an isomorphism
$f^* \mathcal{A} \cong \mathcal{E}nd_{\mathcal{O}_Y}(\mathcal{E})$;
pulling it back by $[\sigma ]$ for $\sigma \in G$, we get isomorphisms
%%%%%%%%%%%%%%%%
\begin{equation}\label{vectorbundles}
\mathcal{E}nd_{\mathcal{O}_Y}([\sigma ]^*\mathcal{E})\cong [\sigma ]^*f^*\mathcal{A}
=f^*\mathcal{A} \cong \mathcal{E}nd_{\mathcal{O}_Y}(\mathcal{E}) .
\end{equation} 
%%%%%%%%%%%%%%%
By Morita theory \cite[IV, Proposition 1.3]{KO},
there is a line bundle $\mathcal{L}_\sigma $ on $Y$ and an isomorphism
$[\sigma ]^*\mathcal{E} \cong \mathcal{E}\otimes \mathcal{L}_\sigma$
which induces the isomorphism 
$\mathcal{E}nd_{\mathcal{O}_Y}([\sigma ]^*\mathcal{E})
\cong \mathcal{E}nd_{\mathcal{O}_Y}(\mathcal{E})$;
moreover the choice of $\mathcal{L}_\sigma$ and the isomorphism 
$[\sigma ]^*\mathcal{E} \cong \mathcal{E}\otimes \mathcal{L}_\sigma$
is unique up to a unique isomorphism.
The mapping $(G \ni \sigma \mapsto \mathcal{L}_\sigma \in \mathrm{Pic}(Y))$
is, therefore, a $1$-cocycle of the $G$-module $\mathrm{Pic}(Y)$.
We define $\phi _{G,Y/X}(\omega )$ to be the element it represents in $H^1(G, \mathrm{Pic}(Y))$.
It can be checked that this element does not depend on the choices of
$\mathcal{A}$ and $\mathcal{E}$.

\vspace{10pt}
When we have a diagram of compatible actions
%%%
%%%%%%
\begin{equation}
\begin{array}{cccccccccc}
G'&\xrightarrow{\rho }&G\\
\circlearrowright&&\circlearrowright\\
Y'&\xrightarrow{r'}&Y\\
\downarrow&&\downarrow\\
X'&\xrightarrow{r}&X
\end{array}
\end{equation}
%%%%%%
%%%
(i.e. the group $G'$ acts on $Y'$ over $X'$, the group $G$ acts on $Y$ over $X$
and the maps $r$, $r'$ and $\rho$ are compatible in the obvious sense),
we have a commutative diagram of groups
%
%%%%%%%%%%%%%%%%
\begin{equation}
\begin{array}{cccccc}
\mathrm{Br}_{\text{Az}}(Y'/X')&\leftarrow &\mathrm{Br}_{\text{Az}}(Y/X)\\
&&\\
{}_{\phi _{G',Y'/X'}} \downarrow &&\downarrow 
_{\phi _{G,Y/X}}\\
&&\\
H^1(G',\mathrm{Pic}(Y'))&\leftarrow &H^1(G,\mathrm{Pic}(Y))
\end{array}
\end{equation}
%%%%%%%%%%%%%
%
where the horizontal maps are the natural functorial ones.

\subsubsection{}\label{twophi}
Suppose $f:Y\to X$ is a Galois covering and an abstract group $G$ acts on
$Y$ through a group homomorphism $G\to \mathrm{Gal}(Y/X)$.
Then we have a commutative diagram
%
%%%%%%%%%%%%%%
\begin{equation}
\begin{array}{ccccccccc}
\mathrm{Br}_{\text{Az}}(Y/X)&\subset &\mathrm{Br}(Y/X)\\
&&\\
{}^{\phi _{G,Y/X}}\downarrow &&\downarrow 
^{\phi _{Y/X}}\\
&&\\
H^1(G,\mathrm{Pic}(Y))&\leftarrow&H^1(\mathrm{Gal}(Y/X),\mathrm{Pic}(Y)) .
\end{array}
\end{equation}
%%%%%%%%%%%%%
%

\section{Main Theorem}\label{sec_th1}
%%%
%%%%%%%%%%%%%%%%%
\begin{theorem}\label{th1}
Let $K$ be a Henselian discrete valuation field of characteristic $0$
with finite residue field $k$,
$X$ be a smooth projective variety over $K$.
Suppose $X$ admits a smooth projective model $\mathcal{X}$ over the integer ring
$\mathcal{O}_K$ whose Picard scheme 
$\mathcal{\bf Pic}_{\mathcal{X}/\mathcal{O}_K}$ over $\mathrm{Spec}(\mathcal{O}_K)$
is smooth.

Then the cokernel of the map 
\[\mathrm{alb}_X: \mathrm{CH}_0(X)^0\to \mathrm{Alb}_X(K) \] 
is canonically isomorphic to the Pontryagin dual of
\[\mathrm{coker}\left( \mathrm{Pic}(\overline{X})^{G_K}
\to\mathrm{NS}(\overline{X})^{G_K} \right) .\]
\end{theorem}
%%%%%%%%%%%%%%%%%%
%%%

%
%%%%%%%%%%%
\begin{remark}\label{rem_th1}
\begin{enumerate}[(1)]
\item The Picard scheme exists in this situation 
(see for example \cite[Theorem 4.8]{kleiman}).

\item By Chebotarev density theorem, the special fiber of $\mathcal{X}$ has a degree 1 zero-cycle.
By the smoothness of $\mathcal{X}$ and the fact that $K$ is Henselian, it lifts to a degree 1 zero-cycle on $X$.
From this, we find that the restriction map
$\mathrm{Br}(K)\to \mathrm{Br}(X)$ is injective and
that its image and the image of the injection 
$\mathrm{Br}(\mathcal{X})\hookrightarrow \mathrm{Br}(X)$
has a trivial intersection.
\end{enumerate}
\end{remark}
%%%%%%%%%%%
%

\paragraph{}
The proof of Theorem \ref{th1} occupies the rest of 
\S \ref{sec_th1},
throughout which we keep the notation in the theorem.

In addition, we use the following notation:
\[ Z=\mathcal{X}\times _{\mathcal{O}_K}k; \]
$\overline{X}$ (resp.~$\overline{Z}$)
= the base change to an algebraic closure of the base field;
\[ \overline{\mathcal{X}}
=\mathcal{X}\otimes _{\mathcal{O}_K}\mathcal{O}_{\overline{K}}. \]

\vspace{10pt}
\begin{remark}\label{cor}
Under the hypothesis of Theorem \ref{th1}, the group
\[ \mathrm{coker}\left( \mathrm{Pic}(\overline{X})^{G_K}
\to \mathrm{NS}(\overline{X})^{G_K} \right)\] 
is a subquotient of $\mathrm{NS}(\overline{X})^{G_K}\{ p \}$,
 the $p$-primary torsion subgroup of
$\mathrm{NS}(\overline{X})^{G_K}$.
Indeed,
by Kummer sequence, for each $n$ prime to $p$ 
we have a commutative diagram
%%%%%%%%%%%%%%%%%%%%%%%%%%%%
\[ \begin{array}{cccccc}%%%%
\mathrm{NS}(\overline{X})/n &\hookrightarrow 
&H^2_{\mathrm{et}}(\overline{X},\mathbb{Z}/n(1)) \\%
\downarrow \mathrm{sp}&&\cong \uparrow \mathrm{cosp}\\%
\mathrm{NS}(\overline{Z})/n &\hookrightarrow 
& H^2_{\mathrm{et}}(\overline{Z},\mathbb{Z}/n(1))
\end{array}\]%%%%%%%%%%%%%%%
%%%%%%%%%%%%%%%%%%%%%%%%%%%%
where the map $\mathrm{cosp}$ is an isomorphism
by proper smooth base change theorem for \'etale cohomology.
From this we see that the vertical map $\mathrm{sp}$ is injective;
it follows the kernel $P$ of the specialization map 
$\mathrm{NS}(\overline{X})\to \mathrm{NS}(\overline{Z})$
is a $p$-primary torsion abelian group. 
On the other hand we have a commutative diagram
\[ \begin{array}{rll}%%%%%%%%%%
&& P^{G_K} \\
&&\rotatebox{90}{$\hookleftarrow$}\\
\mathrm{Pic}(X) &\twoheadrightarrow \mathrm{NS}(X) \hookrightarrow 
&\mathrm{NS}(\overline{X})^{G_K}\\
(*)\rotatebox{90}{$\twoheadleftarrow$}&
&\downarrow\\
\mathrm{Pic}(Z) &\twoheadrightarrow \mathrm{NS}(Z) \xrightarrow[(**)]{\cong} 
&\mathrm{NS}(\overline{Z})^{G_k} \quad .
\end{array}\]%%%%%%%%%%%%
The specialization map $(*)$ is surjective because the Picard scheme is assumed to be smooth
and by Hensel's lemma.
The map $(**)$ is surjective because of $H^1(k,\mathrm{Pic}^0(\overline{Z}))=0$
(Lang's theorem).
From this diagram we see that $\mathrm{NS}(X)+P^{G_K}=\mathrm{NS}(\overline{X})^{G_K}$.
Therefore $\mathrm{NS}(\overline{X})^{G_K}/\mathrm{NS}(X)$ is a quotient of $P^{G_K}$.
\end{remark}

%%%%%%%%%%%%%%%%%%%%%%%%%%%%%%%%%%%%%%
\subsection{Commutative diagrams}%%%%%
%%%%%%%%%%%%%%%%%%%%%%%%%%%%%%%%%%%%%%

We begin the proof of Theorem \ref{th1}.
By the fact $H^3(K,\mathbb{G}_m)=0$, the homomorphism
$\phi _{\overline{X}/X}$ in \S \ref{auto_gal} induces an isomorphism
\[ \phi _{\overline{X}/X}: 
\mathrm{Br}(\overline{X}/X)/\mathrm{Br}(K) 
\xrightarrow{\cong} H^1(K,\mathrm{Pic}(\overline{X}) ). \]
Denote by $\phi$ the following composite map:
%%%%%%%%%%%%%%%%%%%%%
\[\begin{array}{rcl}
H^1(K,\mathrm{Pic^0(\overline{X})})
\to&H^1(K,\mathrm{Pic}(\overline{X}))&\\
&\cong \downarrow \phi _{\overline{X}/X}^{-1} &\\
&\mathrm{Br}(\overline{X}/X)/\mathrm{Br}(K)
&\to \mathrm{Br}(X)/[\mathrm{Br}(K)+\mathrm{Br}(\mathcal{X})]
\end{array}\]
%%%%%%%%%%%%%%%%%%%%

%PPPPPPPPPPPPPPPPPP
\begin{proposition}[(proved in \S \ref{sec_pf_comm})]\label{prop_comm}
The Brauer--Manin pairing and the Tate pairing are compatible
via the Albanese map and the homomorphism $\phi$:
\[ \begin{array}{cccccccccc}
\mathrm{CH}_0(X)^0&\times 
&\frac{\mathrm{Br}(X)}{\mathrm{Br}(K)+\mathrm{Br}(\mathcal{X})}
&\to &\mathbb{Q}/\mathbb{Z}&& \mathrm{ (BM)}\\
\\
\mathrm{alb}_X\downarrow &&\uparrow \phi &&\| \\
\\
\mathrm{Alb}_X(K)&\times &H^1(K,\mathrm{Pic}^0(\overline{X}))
&\to &\mathbb{Q}/\mathbb{Z} &&\mathrm{ (T)}.
\end{array} \]
\end{proposition}
%PPPPPPPPPPPPPPPPPP

%The Brauer--Manin pairing (BM) is known to be non-degenerate on the right
%\cite[Theorem 1.1.3, Remark 2.1.2]{ss}
Here, the Tate pairing (T) is known to be a perfect pairing of
a compact group and a torsion group \cite{WC}.
On the other hand, Saito and Sato have recently proved:
%TTTTTTTTTTTTTTTTTTTTTTT
\begin{theorem}[{(\cite[Theorem 1.1.3]{ss}, applicable by [{\it loc.cit.} Remark 2.1.2])}]
The pairing (BM) is non-degenerate on the right.
\end{theorem}
%TTTTTTTTTTTTTTTTTTTTTTTT

Therefore we conclude:
%CCCCCCCCCCCCCCCCCCCCCC
\begin{corollary}
The group $\mathrm{coker}(\mathrm{alb}_X)$
is the Pontryagin dual of $\ker (\phi )$.
\end{corollary}
%CCCCCCCCCCCCCCCCCCCCCCCCC

By the construction of $\phi$,
we have an exact sequence
\begin{equation}\label{theta}
\begin{array}{rllll}%%%%%%
0 \to &\ker \left( H^1(K,\mathrm{Pic}^0(\overline{X}) ) 
\to H^1(K,\mathrm{Pic}(\overline{X})) \right) \\
\to &\ker \phi 
\xrightarrow[\theta]{} \ker \left(
\mathrm{Br}(\overline{X}/X)/\mathrm{Br}(K)
\to \frac{\mathrm{Br}(X)}{\mathrm{Br}(K)+\mathrm{Br}(\mathcal{X})}
\right)
\end{array} \end{equation}%%%%%%

We have $\ker \left( H^1(K,\mathrm{Pic}^0(\overline{X}) ) 
\to H^1(K,\mathrm{Pic}(\overline{X})) \right) =
\mathrm{coker}\left( \mathrm{Pic}(\overline{X})^{G_K}\to
\mathrm{NS}(\overline{X})^{G_K} \right)$
by the short exact sequence
\[ 0 \to \mathrm{Pic}^0(\overline{X}) \to \mathrm{Pic}(\overline{X}) 
\to \mathrm{NS}(\overline{X}) \to 0. \]
On the other hand we have 
(recall $\overline{\mathcal{X}}:=
\mathcal{X}\otimes _{\mathcal{O}_{K} }  {\mathcal{O}_{\overline{K}} }$)
\[ \begin{array}{rcl}%%%%%%% 
&&\ker \left(  
\frac{\mathrm{Br}(\overline{X}/X)}{\mathrm{Br}(K)}
\to \frac{\mathrm{Br}(X)}{\mathrm{Br}(K)+\mathrm{Br}(\mathcal{X})}
\right) \\
&=&\mathrm{Br}(\overline{X}/X)\cap \mathrm{Br}(\mathcal{X}) 
\text{ (by Remark \ref{rem_th1}(2))} \\
&=&\mathrm{Br}(\overline{\mathcal{X}}/\mathcal{X}),
\end{array} \]%%%%%%%
where the last equality holds because if an element
of $\mathrm{Br}(\mathcal{X})$ is annihilated in $\mathrm{Br}(X_{K'})$
for some finite extention field $K'$ of $K$,
it is annihilated in $\mathrm{Br}(\mathcal{X}_{\mathcal{O}_{K'}})$ too
because the map 
$\mathrm{Br}(\mathcal{X}_{\mathcal{O}_{K'}})\to \mathrm{Br}(X_{K'})$ 
is injective as $\mathcal{X}_{\mathcal{O}_{K'}}$ is
regular by the smoothness assumption on $\mathcal{X}$.

Therefore the sequence (\ref{theta}) takes the form:
\[ 0\to \frac{\mathrm{NS}(\overline{X})^{G_K} }
{\mathrm{Im ~Pic}(\overline{X})^{G_K}}
\to \ker \phi \xrightarrow{\theta }
\mathrm{Br}(\overline{\mathcal{X}}/\mathcal{X}) \]
and hence we are reduced to showing $\theta =0$.

Denote by $\mathrm{sp}$ the specialization maps 
$\mathrm{Pic}(\overline{X})\to \mathrm{Pic}(\overline{Z})$,
$\mathrm{NS}(\overline{X})\to \mathrm{NS}(\overline{Z})$
and the maps induced on their cohomology groups.

%PPPPPPPPPPPPPPPPPPPPPP
\begin{proposition}
The following diagram commutes.
\begin{equation}\label{pentagon}
\xymatrix{
H^1(K,\mathrm{Pic}^0(\overline{X}))\ar[rr]
&
&H^1(K,\mathrm{Pic}(\overline{X}))\ar[r]
&H^1(K,\mathrm{NS}(\overline{X}))\ar[d]^{\mathrm{sp}}
\\%
\ker \phi \ar@{}|{\cup}[u] \ar[r]^{\theta }
&\mathrm{Br}(\overline{\mathcal{X}}/\mathcal{X})
\ar@{}|{\subset }[r] \ar[d]
&\mathrm{Br}(\overline{X}/X)/\mathrm{K} 
\ar[u]^{\phi _{\overline{X}/X}}_{\cong }
&H^1(K,\mathrm{NS}(\overline{Z}))
\\%
{}
&\mathrm{Br}(\overline{Z}/Z) 
\ar[r]^(0.4){\phi _{\overline{Z}/Z}}_(0.4){\cong }
&H^1(k,\mathrm{Pic}(\overline{Z})) \ar[r]
&H^1(k,\mathrm{NS}(\overline{Z})) 
\ar[u]_{\mathrm{inf}}
} \end{equation}
\end{proposition}
%PPPPPPPPPPPPPPPPPPPPPPPP
\begin{proof}
The left square commutes by the definition of $\theta$.
For the right, we use the map 
$\phi _{G_K,\overline{\mathcal{X}}/\mathcal{X} }
:\mathrm{Br}(\overline{\mathcal{X}}/\mathcal{X})
\to H^1(K,\mathrm{Pic}(\overline{\mathcal{X}}))$.
It was defined in \S \ref{sec_auto} on the subgroup
$\mathrm{Br}_{Az}(\overline{\mathcal{X}}/\mathcal{X})
\subset \mathrm{Br}(\overline{\mathcal{X}}/\mathcal{X})$
and by Theorem \ref{gabthm}, this subgroup is in fact
equal to the entire group.

The commutativity of the following diagram
is obvious in the right half.
In the left half it follows from \S \ref{sec_auto}:

\[ \begin{array}{crcccccccc}%%%
\mathrm{Br}(\overline{X}/X)
&\xrightarrow{\hspace{1cm} \phi _{\overline{X}/X} \hspace{2cm}}
&H^1(K,\mathrm{Pic}(\overline{X}))
&\to 
&H^1(K,\mathrm{NS}(\overline{X})) \\
\cup&\nearrow&
\downarrow \mathrm{sp}
&& \downarrow \mathrm{sp}\\
\mathrm{Br}(\overline{\mathcal{X}}/\mathcal{X})
&\xrightarrow{\phi _{G_K,\overline{\mathcal{X}}/\mathcal{X} } }
H^1(K,\mathrm{Pic}(\overline{\mathcal{X}})) \to
&H^1(K,\mathrm{Pic}(\overline{Z}))
&\to&H^1(K,\mathrm{NS}(\overline{Z})) \\
\downarrow&&\uparrow \mathrm{inf}
&& \uparrow \mathrm{inf} \\
\mathrm{Br}(\overline{Z}/Z)
&\xrightarrow{\hspace{2cm}\phi _{\overline{Z}/Z} \hspace{1cm}}
&H^1(k,\mathrm{Pic}(\overline{Z}))
&\to 
&H^1(k,\mathrm{NS}(\overline{Z})) \\
\end{array} \]%%%

This shows the commutativity of the right half of the diagram (\ref{pentagon}).
\end{proof}

\subsection{Injectivity}

In the diagram (\ref{pentagon}), the upper path from
$\ker \phi$ to $H^1(K,\mathrm{NS}(\overline{Z}))$
is zero. Therefore in order to prove $\theta =0$,
it suffices to prove that the composite map from
$\mathrm{Br}(\overline{\mathcal{X}}/\mathcal{X})$
to $H^1(K,\mathrm{NS}(\overline{Z}))$
along the lower path is injective.

%PPPPPPPPPPPPPPPPPPPP
\begin{proposition}
Each of the maps 
\[ H^1(k,\mathrm{Pic}(\overline{Z}))\to 
H^1(k,\mathrm{NS}(\overline{Z}))
\xrightarrow{\mathrm{inf}} H^1(K,\mathrm{NS}(\overline{Z})) \]
is injective.
\end{proposition}
%PPPPPPPPPPPPPPPPPPPP
\begin{proof}
The first one is injective because of Lang's theorem
$H^1(k,\mathrm{Pic}^0(\overline{Z}))=0$.
The fact that the second one is injective follows from the general
fact that the inflation map of group cohomology is always injective
on $H^1$.
\end{proof}
\vspace{10pt}

For a finite extension $K'/K$, we denote the residue field of $K'$
by $k'$ and set
$\mathcal{X}'=\mathcal{X}\otimes _{\mathcal{O}_K}\mathcal{O}_{K'}$
and $Z'=Z \otimes _k k'$.

Since we have
%%%%
\[ \mathrm{Br}(\overline{\mathcal{X}}/\mathcal{X})
=\bigcup _{K'}
\mathrm{Br}(\mathcal{X}'/{\mathcal{X}}) \]
%%%%
where $K'$ runs through all the finite subextensions of $\overline{K}/K$, for proving the injectivity of
$\mathrm{Br}(\overline{\mathcal{X}}/\mathcal{X})
\to \mathrm{Br}(\overline{Z}/Z)$
it suffices to prove the injectivity of
$\mathrm{Br}(\mathcal{X}'/\mathcal{X})\to
\mathrm{Br}(Z'/Z)$
for each $K'/K$.

%
%%%
\begin{proposition}\label{prop_inj}
Let $K'$ be a finite extension of $K$.
For each integer $i\ge 1$, set $Z_i=\mathcal{X}\otimes _{\mathcal{O}_K} \mathcal{O}_K/\mathfrak{m}_K^i$ and
$Z'_i=Z_i \otimes _{\mathcal{O}_K}\mathcal{O}_{K'}$.
Then the following hold:
\begin{enumerate}[(i)]
\item The map $\mathrm{Br}(\mathcal{X})\to \varprojlim _i \mathrm{Br}(Z_i)$ induced by the restriction maps is injective.
\item The map induced by the projection to the first component
\[ \ker \left( \varprojlim _i \mathrm{Br}(Z_i)\to \varprojlim _i \mathrm{Br}(Z'_i)\right)
\to \ker \left( \mathrm{Br}(Z) \to \mathrm{Br}(Z'_1) \right)
\subset  \mathrm{Br}(Z'/Z) \]
is injective.
\item $\mathrm{Br}(\mathcal{X}'/{\mathcal{X}})
\to \mathrm{Br}(Z'/Z)$
is injective.
\end{enumerate}
\end{proposition}
%%%
%
\begin{proof}
(iii) follows from (i),(ii) and the next diagram.
\[  \begin{array}{ccccccc}
\mathrm{Br}(\mathcal{X})&\to&\varprojlim _i \mathrm{Br}(Z_i)&\to&\mathrm{Br}(Z)\\
\downarrow &&\downarrow&&\downarrow\\
\mathrm{Br}(\mathcal{X}')&\to&\varprojlim _i \mathrm{Br}(Z'_i)&\to&\mathrm{Br}(Z')
\end{array} \]

To prove (i),
we use
%%%%%%%%%%%%%%%%%%%%%%%%
\begin{lemma}[{\cite[III, Lemme (3.3)]{BR}}]
Let $f:\mathcal{X}\to Y$ be a proper flat morphism, with $Y$ the spectrum of
an excellent Henselian discrete valuation ring. Let $Y=\mathrm{Spec}(A)$, $\mathfrak{m}$
be the maximal ideal of $A$, $Z_i := Z\otimes _A A/\mathfrak{m}^i$.
Suppose further that the projective system $\left( \mathrm{Pic}(Z_i) \right) _i$
satisfies the Mittag-Leffler condition. Then the canonical morphism
\[
\mathrm{Br_{Az}}(\mathcal{X}) \to \varprojlim _i \mathrm{Br_{Az}}(Z_i)
\]
is injective.
\end{lemma}
%%%%%%%%%%%%%%%%%%%%%%%%%%
We can apply this lemma to our situation $\mathcal{X}\to \mathrm{Spec}(\mathcal{O}_K)$;
the Mittag-Leffler condition is fulfilled by the smoothness assumption on the Picard scheme.

Therefore in the diagram:
\[
\begin{array}{cccccccc}
\mathrm{Br}_{Az}(\mathcal{X})&\xrightarrow{(*)}&\varprojlim _i \mathrm{Br_{Az}}(Z_i)\\
\rotatebox{-90}{$\subset$}_{(**)}&&\rotatebox{-90}{$\subset$}\\
\mathrm{Br}(\mathcal{X})&\xrightarrow{(***)} &\varprojlim _i\mathrm{Br}(Z_i)
\end{array}
\]
the map $(*)$ is injective and the inclusion $(**)$ is in fact an equality by Theorem \ref{gabthm}.
It follows that the map $(***)$ is also injective.
This proves Proposition \ref{prop_inj}(i).

We show (ii).
From the exact sequence of sheaves on the topological space $Z$:
\[\begin{array}{rcl}%%%%%%%%
0 \to &\mathfrak{m}^i\mathcal{O}_{Z_{i+1}}& \xrightarrow{1+} 
\mathcal{O}^*_{Z_{i+1}} \to \mathcal{O}^*_{Z_i}\to 1\\
&\rotatebox{90}{$\cong$}&\\
&\mathcal{O}_Z&
\end{array}\]%%%%%%%%%
where the isomorphism is specified if we choose a generator of 
$\mathfrak{m}$.
There is a similar one on $Z'$:
\[\begin{array}{rcl}%%%%%%%%
0 \to &\mathfrak{m}^i\mathcal{O}_{Z'_{i+1}}& \xrightarrow{1+} 
\mathcal{O}^*_{Z'_{i+1}} \to \mathcal{O}^*_{Z'_i}\to 1\\
&\rotatebox{90}{$\cong$}&\\
&\mathcal{O}_{Z'_1}&
\end{array}\]%%%%%%%%%
From these, one obtains a commutative diagram
%%%%%%%%%%
\[\begin{array}{ccccccccc}
\mathrm{Pic}(Z_{i+1})&\xrightarrow{\psi _1} &\mathrm{Pic}(Z_i)&\xrightarrow{\psi _2} 
&H^2(Z,\mathcal{O}_Z)
&\to &\mathrm{Br}(Z_{i+1})&\to &\mathrm{Br}(Z_i) \\
\downarrow &&\downarrow &&\downarrow v&&\downarrow &&\downarrow \\
\mathrm{Pic}(Z'_{i+1})&\xrightarrow{\psi '_1} &\mathrm{Pic}(Z'_i)&\xrightarrow{\psi '_2} 
&H^2(Z'_1,\mathcal{O}_{Z'_1})
&\to &\mathrm{Br}(Z'_{i+1})&\to &\mathrm{Br}(Z'_i) .
\end{array}\]
%%%%%%%%%%
The maps $\psi _1$ and $\psi '_1$ are surjective by the assumption 
that the Picard scheme is smooth.
Therefore $\psi _2$ and $\psi '_2$ are zero maps.
The vertical map $v$ is injective by flat base change theorem for coherent cohomology.
So the group
$\mathrm{Br}(Z'_{i+1}/Z_{i+1})$
injects into $\mathrm{Br}(Z'_{i}/Z_{i})$, 
hence into $\mathrm{Br}(Z'_{1}/Z)$.

This complete the proof.
\end{proof}

In order to prove Theorem \ref{th1}, it remains to show Proposition \ref{prop_comm}.
This is done in {\S}\ref{sec_pf_comm}.

\begin{remark}
Proposition \ref{prop_inj}(iii) implies the relation
$\mathrm{Br}(\mathcal{X}'/\mathcal{X})\subset \mathrm{Br}(\mathcal{X}'^{\text{ur}}/\mathcal{X})$ where $\mathcal{X}'^{\text{ur}}$ denotes the base change of $\mathcal{X}$ to the maximal unramified subextension of $\mathcal{O}_{K'}/\mathcal{O}_K$.
In view of \cite[Theorem 1.1.3, Remark 2.1.2]{ss} again,
this means that any element of 
$\mathrm{Br}(X)$ which vanishes
at every closed point and in $\mathrm{Br}(\overline{X})$ is trivialized by an unramified
extension of $K$.
It would be interesting to ask if this holds 
without hypotheses in Theorem \ref{th1}.
\end{remark}

%%%%%%%%%%%%%%%%%%%%%%%%%%%%%%%%%%%%%%%%%%%%%%%%%%%
\subsection{Proof of Proposition \ref{prop_comm}}\label{sec_pf_comm}%%%%%%%%%%%
%%%%%%%%%%%%%%%%%%%%%%%%%%%%%%%%%%%%%%%%%%%%%%%%%%%%
For schemes $Y$, we use the big \'etale site
$(\mathrm{Sch}/Y)_{\mathrm{et}}$,
whose underlying category is the category of all schemes locally of finite type
over $Y$ and the coverings are \'etale surjective morphisms.
Denote by $D((\mathrm{Sch}/Y)_{\mathrm{et}})$ the derived category of complexes of abelian \'etale sheaves.
%By abuse of notation,
%$\mathrm{Hom}_{(\mathrm{Sch}/Y)_{et}}$,
%$\mathrm{Ext}^i_{(\mathrm{Sch}/Y)_{et}}$, etc.~
%will mean those of abelian sheaves
%on $(\mathrm{Sch}/Y)_{et}$.

%We will use the following properties of this site:
%PPPPPPPPPPPPPPPPPPPPP
%\begin{proposition}[{cf.~\cite[Lemma 1.3]{hamel}}]\label{site}
%\begin{enumerate}[(1)]
%\item For finite morphisms $f: Y' \to Y$, the push-forward functor
%$f_*:(\mathrm{Sm}/Y')_{\mathrm{et}}
%\to (\mathrm{Sm}/Y)_{\mathrm{et}}$ is exact.
%\item Let $\alpha :(\mathrm{Sch}/Y)_{\mathrm{fl}}\to
%(\mathrm{Sm}/Y)_{\mathrm{et}}$ be the canonical map of sites.
%For sheaves $G_1, G_2$ which are represented by smooth group shcemes over $Y$, we have
%\[ R\alpha _*
%R\mathcal{H}om_{(\mathrm{Sch}/Y)_{\mathrm{fl}}}(G_1,G_2)
%=R\mathcal{H}om_{(\mathrm{Sm}/Y)_{\mathrm{et}}}(G_1,G_2)   \]
%In particular for an Abelian variety $A$ with its dual $\Check{A}$,
%we have 
%\[ R\mathcal{H}om_{(\mathrm{Sch}/K)_\text{fl}}(A,\mathbb{G}_m)
%=\Check{A}[-1].\]
%\end{enumerate}
%\end{proposition}
%PPPPPPPPPPPPPPPPPPP
%{\it Proof.}

%Q.E.D.

For $K$-schemes $\varphi :Y\to \mathrm{Spec}K$,
we will consider the groups $\mathrm{Hom}_{D((\mathrm{Sch}/K)_\mathrm{et})}(R\varphi _* \mathbb{G}_m,\mathbb{G}_m)$. They form a covariant functor in $K$-schemes $Y$.
%PPPPPPPPPPPPPPPPPPPPP
\begin{proposition}\label{point}
For any finite field extension 
$\varphi _{K'} :\mathrm{Spec}K' \to \mathrm{Spec}K$,
there is a canonical isomorphism
\[ \mathrm{Hom}_{D((\mathrm{Sch}/K)_{\mathrm{et}})}(R\varphi _{K' *} \mathbb{G}_m
,\mathbb{G}_m)
=\mathbb{Z}. \]
\end{proposition}
%PPPPPPPPPPPPPPPPPPP
\begin{proof}
Since $\varphi _{K'}$ is finite, we have $R\varphi _{K' *} \mathbb{G}_m
= \varphi _{K' *}\mathbb{G}_m=R_{K'/K}\mathbb{G}_m$ (the Weil restriction).
Therefore the group $\mathrm{Hom}_{D((\mathrm{Sch}/K)_\mathrm{et})}(R\varphi _{K' *} 
\mathbb{G}_m,\mathbb{G}_m)$
equals to the group of homomorphisms 
\[ R_{K'/K}\mathbb{G}_m\to \mathbb{G}_m \]
of group schemes over $K$.
Giving such a morphism is equivalent to giving a morphism
over $\overline{K}$
\[ \prod _{\mathrm{Hom}_K(K',\overline{K})}\mathbb{G}_m\to \mathbb{G}_m \]
which is invariant under the action of $G_K$, which acts on 
$\prod \mathbb{G}_m$ by permutation of components 
and trivially on the right hand side.
Since the set of homomorphisms from the $\overline{K}$-group scheme
$\mathbb{G}_m$ to itself is 
canonically isomorphic to
$\mathbb{Z}$, the assertion follows.
Explicitly, $1\in \mathbb{Z}$ corresponds to the
norm map $R_{K'/K}\mathbb{G}_{m,K'}\to \mathbb{G}_{m,K}$.
\end{proof}

Let $\varphi :X\to \mathrm{Spec}K$ as in Theorem \ref{th1}
%be the structure morphism.
%By \cite[Corollary 1.4 and Proposition 3.2]{hamel}, 
%we have a commutative diagram with exact rows
and $i_x: x\to X$ be a closed point.
Write 
$\varphi _x=\varphi \circ i_x: x\to \mathrm{Spec}K$.
From the map $\mathbb{G}_m \to i_{x *}\mathbb{G}_m$
we get a map (the first equality is due to Proposition \ref{point})
\[ \mathbb{Z}=\mathrm{Hom}_{D((\mathrm{Sch}/K)_\mathrm{et})}(R\varphi _{x *} \mathbb{G}_m
,\mathbb{G}_m) \to 
\mathrm{Hom}_{D((\mathrm{Sch}/K)_\mathrm{et})}(R\varphi _* \mathbb{G}_m
,\mathbb{G}_m). \]
Thus we get a map
\[ \mathrm{cl}:Z_0(X) \to \mathrm{Hom}_{D((\mathrm{Sch}/K)_\mathrm{et})}(R\varphi _* \mathbb{G}_m
,\mathbb{G}_m) \]
which is known to factor through $\mathrm{CH}_0(X)$
\cite[Proposition 3.2]{hamel}.

By functoriality we have a commutative diagram
%%%
%%%%%%%%%
\[\begin{array}{cccccc}
Z_0(X)
&\xrightarrow{\deg}
&\mathbb{Z}=Z_0(\mathrm{Spec}K) \\
\mathrm{cl}\downarrow&&\mathrm{cl}\downarrow \rotatebox{90}{$=$}\\
\mathrm{Hom}_{D((\mathrm{Sch}/K)_{\mathrm{et}})}(R\varphi _*\mathbb{G}_m,\mathbb{G}_m)
&
\xrightarrow{\deg}
&\mathbb{Z}=\mathrm{Hom}_{D((\mathrm{Sch}/K)_\mathrm{et})}(\mathbb{G}_m, \mathbb{G}_m)
\end{array}\]
%%%%%%%%%
%%%
So we get an induced map between kernels
\[ \mathrm{cl}:Z_0(X)^0\to 
\mathrm{Hom}_{D((\mathrm{Sch}/K)_{\mathrm{et}})}(\tau _{\ge 1}R\varphi _*\mathbb{G}_m,\mathbb{G}_m)   \]
where we put $Z_0(X)^0=\ker \left( Z_0(X) \xrightarrow{\deg }\mathbb{Z} \right)$.

%where the second row comes from the distinguished triangle in $D((\mathrm{Sm}/K)_{\mathrm{sm}})$
%
%%%
%\[ \mathbb{G}_m \to R\varphi _*\mathbb{G}_m 
%\to \tau _{\ge 1}R\varphi _*\mathbb{G}_m \xrightarrow{+1}. \]
%%%
%
There is a commutative diagram of Yoneda pairings
obtained from maps
\[ \mathbf{Pic}^0_{X/K}\to \mathbf{Pic}_{X/K}
\to \tau _{\ge 1}R\varphi _*\mathbb{G}_m
\leftarrow R\varphi _*\mathbb{G}_m : \]
%%
%%%%%%%%%%%%%%%%%%%%%%%%%%%
\[\begin{array}{cccccc}
&&&H^2_{\mathrm{et}}(K,\mathbb{G}_m)&&\\
&&&\downarrow&&\\
(\mathrm{BM})'&\mathrm{Hom}_{D((\mathrm{Sch}/K)_\mathrm{et})}(R\varphi _*\mathbb{G}_m,\mathbb{G}_m)&\times
&\mathbb{H}^2_\mathrm{et}(K,R\varphi _*\mathbb{G}_m)&\to
&H^2_\mathrm{et}(K,\mathbb{G}_m)\vspace{1mm}\\
&\uparrow&&{~}\downarrow \eta&&\rotatebox{90}{$=$}\vspace{1mm}\\
&\mathrm{Hom}_{D((\mathrm{Sch}/K)_\mathrm{et})}(\tau _{\ge 1}R\varphi _*\mathbb{G}_m,\mathbb{G}_m)&\times
&\mathbb{H}^2_\mathrm{et}(K,\tau _{\ge 1}R\varphi _*\mathbb{G}_m)&\to
&H^2_\mathrm{et}(K,\mathbb{G}_m)\vspace{1mm}\\
&\downarrow&&{~}\uparrow \phi '&&\rotatebox{90}{$=$}\vspace{1mm}\\
&\mathrm{Hom}_{D((\mathrm{Sch}/K)_\mathrm{et})}(\mathbf{Pic}_{X/K}[-1],\mathbb{G}_m)&\times
&\mathbb{H}^2_\mathrm{et}(K,\mathbf{Pic}_{X/K}[-1])&\to&H^2_\mathrm{et}(K,\mathbb{G}_m)\vspace{1mm}\\
&\downarrow&&\uparrow&&\rotatebox{90}{$=$}\vspace{1mm}\\
(\mathrm{T})'
&\mathrm{Hom}_{D((\mathrm{Sch}/K)_\mathrm{et})}(\mathbf{Pic}^0_{X/K}[-1],\mathbb{G}_m)&\times
&\mathbb{H}^2_\mathrm{et}(K,\mathbf{Pic}^0_{X/K})[-1])&\to
&H^2_\mathrm{et}(K,\mathbb{G}_m)
\end{array}\]

%%%%%%%%%%%%%%%%%%%%%%%%%%%
%%
%By the comparison theorem of \'etale, smooth and fppf cohomology in coefficients in
%smooth commutative group schemes \cite[Th\'eor\`eme 11.7]{BR}, this is rewritten as
By 
\begin{gather*}
H^2_\mathrm{et}(K,\mathbb{G}_m)=\mathrm{Br}(K)=\mathbb{Q}/\mathbb{Z}, \\
\mathbb{H}^2_\mathrm{et}(K,R\varphi _*\mathbb{G}_m)
=H^2_\mathrm{et}(X,\mathbb{G}_m)=\mathrm{Br}(X)
\text{ and }\\
\mathrm{Hom}_{D((\mathrm{Sch}/K)_\mathrm{et})}(\mathbf{Pic}^0_{X/K}[-1],\mathbb{G}_m)
=\mathrm{Alb}_X(K), \end{gather*}
this is rewritten as

%%
%%%%%%%%%%%%%%%%%%%%
\[\begin{array}{cccccc}
&&&\mathrm{Br}(K)&&\\
&&&\downarrow&&\\
(\mathrm{BM})'&\mathrm{Hom}_{D((\mathrm{Sch}/K)_\mathrm{et})}(R\varphi _*\mathbb{G}_m,\mathbb{G}_m)&\times
&\mathrm{Br}(X)&\to
&\mathbb{Q}/\mathbb{Z}\vspace{1mm}\\
&\uparrow&&{~}\downarrow \eta&&\rotatebox{90}{$=$}\vspace{1mm}\\
(\mathrm{BM})''
&\mathrm{Hom}_{D((\mathrm{Sch}/K)_\mathrm{et})}(\tau _{\ge 1}R\varphi _*\mathbb{G}_m,\mathbb{G}_m)&\times
&\mathbb{H}^2_\mathrm{et}(K,\tau _{\ge 1}R\varphi _*\mathbb{G}_m)&\to
&\mathbb{Q}/\mathbb{Z}\vspace{1mm}\\
&\downarrow&&{~}\uparrow \phi '&&\rotatebox{90}{$=$}\vspace{1mm}\\
&\mathrm{Hom}_{D((\mathrm{Sch}/K)_\mathrm{et})}(\mathrm{Pic}(\overline{X})[-1],\mathbb{G}_m)&\times
&H^1_{\mathrm{Gal}}(K,\mathrm{Pic}(\overline{X}))&\to&\mathbb{Q}/\mathbb{Z}\vspace{1mm}\\
&\downarrow&&\uparrow&&\rotatebox{90}{$=$}\vspace{1mm}\\
(\mathrm{T})'
&\mathrm{Alb}_X(K)&\times
&H^1_{\mathrm{Gal}}(K,\mathrm{Pic}^0(\overline{X}))&\to
&\mathbb{Q}/\mathbb{Z}
\end{array}\]

%%%%%%%%%%%%%%%%%%%
%%
The pairing $(\mathrm{BM})'$ can be seen to give the Brauer-Manin pairing when composed
with the map 
$\mathrm{CH}_0(X)\to \mathrm{Hom}_{D((\mathrm{Sch}/K)_{\mathrm{et}})}(R\varphi _*\mathbb{G}_m,\mathbb{G}_m)$.
%(cf.~\cite[Lemma 3.1]{hamel}).
The map $\eta$ is surjective because $H^3_{\mathrm{Gal}}(K,\mathbb{G}_m)=0$, so $\eta$ induces an isomorphism
\begin{equation}\label{brbr}%%%%%%%%
 \mathbb{H}^2_\mathrm{et}(K,\tau _{\ge 1}R\varphi _*\mathbb{G}_m)
=\mathrm{Br}(X)/\mathrm{Br}(K).
\end{equation}%%%%%%%%%%%%%%%%%%%%%%%
Hence the pairing $(\mathrm{BM})''$,
when composed with the map 
%%
%%%%%%%%%%%%%
\[ Z_0(X)^0\to 
\mathrm{Hom}_{D((\mathrm{Sch}/K)_{\mathrm{et}})}(\tau _{\ge 1}R\varphi _*\mathbb{G}_m,\mathbb{G}_m),\]
%%%%%%%%%%%%%%
%%
gives the Brauer-Manin pairing
\[ Z_0(X)^0\times \mathrm{Br}(X) / \mathrm{Br}(K) \to \mathbb{Q}/\mathbb{Z}. \]%%%%%%
The pairing $(\mathrm{T})'$ is the same as the Tate pairing as explained in
\cite[Remark 3.5 in Chapter I]{mil}.

Thus we get the commutative diagram:
%%%%%%
\begin{equation}\label{pairings}
\begin{array}{crcccccc}
(\mathrm{BM})
&Z_0(X)^0&\times
&\mathrm{Br}(X)/\mathrm{Br}(K)&\to
&\mathbb{Q}/\mathbb{Z}\\
&&&{~}\uparrow \phi '&&\rotatebox{90}{$=$}\\
&a\downarrow&
&H^1_{\mathrm{Gal}}(K,\mathrm{Pic}(\overline{X}))&\to&\mathbb{Q}/\mathbb{Z}\\
&&&\uparrow&&\rotatebox{90}{$=$}\\
(\mathrm{T})
&\mathrm{Alb}_X(K)&\times
&H^1_{\mathrm{Gal}}(K,\mathrm{Pic}^0(\overline{X}))&\to
&\mathbb{Q}/\mathbb{Z}
\end{array}\end{equation}

%%%%%%
We have to show that the map $a$ is equal to the
Albanese map
and the map $\phi '$ is the inverse of
$\phi _{\overline{X}/X}$.
The latter can be checked if one notices that 
the Hochschild-Serre spectral sequence is realized as the one associated
with the tower
\[  R\Gamma (\overline{X},\mathbb{G}_m)\to 
\tau _{\ge 1}R\Gamma (\overline{X},\mathbb{G}_m)\to 
\tau _{\ge 2}R\Gamma (\overline{X},\mathbb{G}_m)\to \dots  \]
in the derived category of
$G_K$-modules, together with triangles
\[ R^{n}\Gamma (\overline{X},\mathbb{G}_m)[-n]
\to \tau _{\ge n}R\Gamma (\overline{X},\mathbb{G}_m)
\to \tau _{\ge n+1}R\Gamma (\overline{X},\mathbb{G}_m)\xrightarrow{+1}.  \]

\subsubsection{}%The Albanese map}

We will check that the map
\[ a:Z_0(X)^0\to \mathrm{Hom}_{D\left( (\mathrm{Sch}/K)_{\mathrm{et}}\right) }(\mathbf{Pic}_X^0[-1],\mathbb{G}_m)
=\mathrm{Alb}_X(K)  \]
in (\ref{pairings})
is equal to the Albanese map $\mathrm{alb}_X$.

Note that for any field extension $K'/K$
the map $\mathrm{Alb}_X(K)\to \mathrm{Alb}_X(K')$
is injective, and hence it suffices to show the equality after an arbitrary field extension.
In particular we can assume $X$ is given a base point
$x_0\in X(K)$.

There is a sheafified version of the map $a$.
Suppose we are given two morphisms
$f_1,f_2$ from a $K$-scheme $Y$ to $X$.
Since the composite
$Y\xrightarrow{\Gamma _{f_i}} Y\times X \xrightarrow{\pi :=\mathrm{pr}_1}Y ~(i=1,2)$
is the identity morphism,
the composite
$\mathbb{G}_{m,Y}\xrightarrow{\pi ^\# }R\pi _*\mathbb{G}_{m,Y\times X}
\xrightarrow{\Gamma _i^\# } \mathbb{G}_{m,Y}$
is the identity morphism.
Note that $\pi ^\# $ induces an isomorphism
$\mathbb{G}_{m,Y}\cong \pi _*\mathbb{G}_{m,Y\times X}$.Therefore the map
$\Gamma _{f_1}^\# -\Gamma _{f_2}^\# :R\pi _*\mathbb{G}_{m,Y\times X}\to \mathbb{G}_{m,Y}$
induces a map
$\Gamma _{f_1}^\# -\Gamma _{f_2}^\# :\tau _{\ge 1}R\pi _*\mathbb{G}_{m,Y\times X}\to \mathbb{G}_{m,Y}$.
Composed with
$(\mathbf{Pic}_X^0)_Y [-1]\hookrightarrow
R^1\pi _* \mathbb{G}_{m,Y\times X} [-1]\to \tau _{\ge 1}R\pi _* \mathbb{G}_{m,Y\times X}$,
it gives 
%%%%
\[ (\mathbf{Pic}_X^0)_Y[-1]\to \mathbb{G}_{m,Y}  \]
%%%
i.e.\ an element of 
$\mathrm{Ext}^1_{(\mathrm{Sch}/Y)_{\mathrm{et}}}
((\mathbf{Pic}_X^0)_Y,\mathbb{G}_{m,Y})$.
Thus we have defined a map of sets
%%%
\[ (X\times _K X)(Y)\to \mathrm{Ext}^1_{(\mathrm{Sch}/Y)_{\mathrm{et}}}
((\mathbf{Pic}_X^0)_Y,\mathbb{G}_{m,Y}).   \]
%%%
(Here, $\mathrm{Ext}^1_{(\mathrm{Sch}/Y)_{\mathrm{et}}}(-,-)$ denotes the $\mathrm{Ext}^1$ computed in the category of abelian sheaves on $(\mathrm{Sch}/Y)_{\mathrm{et}}$. The same applies below.)
These are organized to define a map of sheaves
%%%
\[ b: X\times _K X\to \mathcal{E}xt^1_{(\mathrm{Sch}/K)_{\mathrm{et}}}
(\mathbf{Pic}_X^0,\mathbb{G}_{m})(=\mathrm{Alb}_X)  \]
%%%
(resp.\ $b_{*}: X\to \mathrm{Alb}_X$ if $X$ is pointed, putting
$f_2=$ the constant map to $x_0$).

To prove the claimed equality of maps,
it suffices to show that the map
$b:X\times _K X \to \mathrm{Alb}_X$ just defined
is the Albanese map.
By the universality of $\mathrm{Alb}_X$,
it suffices to show the map $b_*$ for
$(X,x_0)=(\mathrm{Alb}_X,0)$ is the identity map of 
$\mathrm{Alb}_X$.

Thus the next general proposition is enough for us to conclude.
%%%%%%%
%%%%%%%%%%%%%%%%%%%%
\begin{proposition}\label{albdefA}
Let $A$ be an abelian variety over a field $F$
and $P$ its dual abelian variety.
Then the map
\[ b_*: A\to \mathcal{E}xt^1_{(\mathrm{Sch}/K)_{\mathrm{et}}}(P,\mathbb{G}_m)  \]
(here $A$ is pointed by $0$) is the same map as 
the one induced by the Poincar\'e sheaf $\mathcal{P}$
on $A\times P$.
\end{proposition}
%%%%%%%%%%%%%%%%%%%%
%%%%%%

Here, the Poincar\'e sheaf $\mathcal{P}$ is a biextension of $A\times P$ by $\mathbb{G}_m$
---
a $\mathbb{G}_m$-torsor on $A\times P$
which is given a structure of an extension
of the $A$-group $P_A=A\times P$ by $\mathbb{G}_{m,A}$
%%%
\begin{equation}\label{extA}
 0\to \mathbb{G}_{m,A}\to \mathcal{P} \to P_A\to 0  
 \end{equation}
%%%
and a structure of an extension of the $P$-group $A_P=A\times P$ by $\mathbb{G}_{m,P}$
%%%
\begin{equation}\label{extP} 
0\to \mathbb{G}_{m,P}\to \mathcal{P} \to A_P\to 0   
\end{equation}
%%%
---
which is characterized by the property that
for each $F$-scheme $Y$ the map 
%%%%%%%%%%%%%%%%%%%%%%%%%%%%
\begin{equation}\label{Pic-sch} 
\begin{array}{ccrccc}%%%%
P(Y)&\to &\mathbf{Pic}_A^0(Y)=\{ L\in \mathrm{Pic}(A\times Y,0\times Y) | \forall y\in Y ~~L_y \text{ is}\\
 &&\text{algebraically equivalent to $0$}  \} \\
{}\\
(f:Y\to P)&\mapsto &\text{pull-back of the $\mathbb{G}_m$-torsor $\mathcal{P}$ on $A\times P$ by the map } \\
&&\text{$\mathrm{id}\times f:A\times Y\to A\times P$}
\end{array}  \end{equation}%%%%%%%%%%%%
%%%%%%%%%%%%%%%%%%%%%%%%%%%
is bijective, where for a scheme $V$ and a closed subscheme $Z$ of $V$, $\mathrm{Pic}(V,Z)$ denotes
the group of isomorphism classes of pairs $(L, \varphi )$, $L$ being a line bundle over $V$
and $\varphi$ being an isomorphism $\varphi :L|_Z
\xrightarrow{\cong} \mathcal{O}_Z$.

In the statement of Proposition \ref{albdefA},
``the map induced by the Poincar\'e sheaf''
refers to the map which associates to
a morphism $f:Y\to A$ the pull-back of
(\ref{extA}) by it.
Let us denote the map by $b_\mathcal{P}$.

The proof of Proposition \ref{albdefA} will be completed in
{\S}\ref{pf-albdef}.

\subsubsection{}\label{Pic-compat}
Here we recall a basic compatibility.
For any scheme $Y$ there are canonical isomorphisms
%%%
\[ H^i_{\mathrm{et}}(Y,\mathbb{G}_m)\cong 
\mathrm{Ext}^i_{(\mathrm{Sch}/Y)_{\mathrm{et}}}(\mathbb{Z}_Y,
\mathbb{G}_{m,Y}) \]
%%%
and the one for $i=1$ is described as follows:
given an extension
%%%
\[ 0\to \mathbb{G}_{m,Y}\to E\to \mathbb{Z}_Y\to 0,  \]
%%%
the corresponding $\mathbb{G}_m$-torsor is given by
the pull-back of $E$ to 
$Y\times \{ 1 \} \subset Y\times \mathbb{Z}$.
Conversely, given a $\mathbb{G}_m$-torsor $E'$ on $Y$,
the corresponding extension is given by
$\coprod _{i\in \mathbb{Z}} E'^i$,
$E'^i$ being the $i$-th power of $E'$ as a $\mathbb{G}_m$-torsor.

\paragraph{}
Now let $Y$ be an arbitrary $F$-scheme and 
$f:Y\to A$ an $F$-morphism.
By pull-back, we have a $\mathbb{G}_m$-torsor
$\mathcal{P}\times _A Y$ on $Y\times P$
and an extension of $Y$-groups
%%%
\begin{equation}\label{extY}
 0\to \mathbb{G}_{m,Y}\to \mathcal{P}\times _AY \to
P_Y\to 0  \end{equation}
%%%
equivalently, a morphism in $D((\mathrm{Sch}/Y)_{\mathrm{et}})$
%%%
\begin{equation}\label{PtoG1}
 P_Y\to \mathbb{G}_{m,Y}[1]   .
 \end{equation}
%%%

Suppose given an $F$-morphism $g:Y\to P$.
One checks the following elements (i)--(iv)
 of $H^1(Y,\mathbb{G}_m)=\mathrm{Ext}^1(\mathbb{Z}_Y,\mathbb{G}_{m,Y})$ are the same:
%%%%%%%%%%%%%%%%%%%%%%%%%%%%
\begin{enumerate}[(i)]%%%%%%
\item The image of $g$ by the connecting homomorphism
\[ P_Y(Y)\to H^1(Y,\mathbb{G}_m) \]
arising from (\ref{extY}).
What is the same, the image of $g$ by the map
%%%
\[ H^0(Y,P_Y)\to H^0(Y,\mathbb{G}_{m,Y})=H^1(Y,\mathbb{G}_{m,Y}) \]
%%%
arising from (\ref{PtoG1}).
\item The image of $(\mathcal{P}_Y,g)$ by the Yoneda pairing
%%%
\[ \mathrm{Ext}^1(P_Y,\mathrm{G}_{m,Y})\times
H^0(Y,P_Y)\to H^1(Y,\mathbb{G}_m).
\]
%%%
What is the same, the image of $(\mathcal{P}_Y,g)$
by the Yoneda pairing
%%%
\[ \mathrm{Ext}^1(P_Y,\mathbb{G}_{m,Y})\times
\mathrm{Hom}_{(\mathrm{Sch}/Y)_{\mathrm{et}}}(\mathbb{Z}_Y,P_Y)\to \mathrm{Ext}^1(\mathbb{Z}_Y,\mathbb{G}_m)  . \]
%%%
%
\item The pull-back of the $\mathbb{G}_m$-torsor 
$\mathcal{P}\times _A Y$ on $Y\times P$ by
$\Gamma _g:Y\to Y\times P$.
What is the same, the pull-back of the extension
(\ref{extY}) by the map $g:\mathbb{Z}_Y\to P_Y$.
\item The pull-back of the $\mathbb{G}_m$-torsor
$\mathcal{P}$ on $A\times P$ by the map 
$f\times g:Y\to A\times P$.
\end{enumerate}%%%%%%%%%%
%%%%%%%%%%%%%%%%%%%%%%%%%

\subsubsection{Proof of Proposition \ref{albdefA}.}
\label{pf-albdef}
Suppose we are given an $F$-scheme $Y$ and a morphism $f:Y\to A$. Factor $f$ as $Y\xrightarrow{\Gamma _f}
A\times Y \xrightarrow{\mathrm{pr}_1}A$
and write $\pi =\mathrm{pr}_2:A\times Y\to Y$.
Over $A$ there is the Poincar\'e sheaf $\mathcal{P}$, 
which is an extension of
$P_A$ by $\mathbb{G}_{m,A}$.
We see it as an element of
$\mathrm{Hom}_{D((\mathrm{Sch}/A)_{\mathrm{et}})}(P_A,\mathbb{G}_{m,A}[1])$.
%%%
\[ \begin{array}{r}
\underset{\vdots}{\mathcal{P}} \\
 Y\overset{\pi :=\mathrm{pr}_2}{\underset{\Gamma _f}{\leftrightarrows}} A\times Y
\xrightarrow{\mathrm{pr}_1} A
\end{array}\]
%%%
We have a commutative diagram in $D((\mathrm{Sch}/Y)_{\mathrm{et}})$
%%%
\begin{equation}\label{P_f}
\begin{array}{cccccc}
P_Y&\xrightarrow{f^*\mathcal{P}}&\mathbb{G}_{m,Y}[1]\\
{}\\
{}^{\Gamma _f^\# } \uparrow&&\uparrow ^{\Gamma _f^\# }\\
R\pi _*P_{A\times Y}&\xrightarrow{R\pi _*(\mathcal{P}\times _F Y)}&R\pi _*\mathbb{G}_{m,A\times Y}[1]\\
{}\\
{}^{\pi ^\# } \uparrow&&\\
P_Y&&
\end{array}
\end{equation}
%%%
and if we put $f=$ the constant map to $0$, it looks like:
%%%
\begin{equation}\label{P_0}
\begin{array}{cccccc}
P_Y&\xrightarrow{0}&\mathbb{G}_{m,Y}[1]\\
{}\\
{}^{\Gamma _0^\# } \uparrow&&\uparrow ^{\Gamma _0^\# }\\
R\pi _*P_{A\times Y}&\xrightarrow{R\pi _*(\mathcal{P}\times _F Y)}&R\pi _*\mathbb{G}_{m,A\times Y}[1]\\
{}\\
{}^{\pi ^\# } \uparrow&&\\
P_Y&&
\end{array}
\end{equation}
%%%

Note that we have $\Gamma _f^\# \circ \pi ^\# =\Gamma _0^\# \circ \pi ^\# =\mathrm{id}:P_Y\to P_Y$.
Therefore by (\ref{P_f}) we have that 
$b_\mathcal{P}(f):=f^* \mathcal{P}\in 
\mathrm{Ext}^1_{(\mathrm{Sch}/Y)_{\mathrm{et}}}
(P_Y,\mathbb{G}_{m,Y})$
is equal to 
$\Gamma _f^\# \circ \mathcal{P}\times _FY 
\circ \pi ^\# $.
From (\ref{P_0}) we see 
$\Gamma _0^\# \circ (\mathcal{P}\times _FY)
\circ \pi ^\# =0
: P_Y\to \mathbb{G}_{m,Y}[1]$.

%CCCCCCCCCCCCCCC
\begin{claim}\label{mou-chotto}
The following commutes, where the unnamed maps are
the canonical ones:
\[ \begin{array}{cccccccc}
R\pi _*\mathbb{G}_{m,A\times Y}[1]&\to 
&\tau _{\ge 1}R\pi _*\mathbb{G}_{m,A\times Y}[1]&
\left(
\xrightarrow{\Gamma _f^\# -\Gamma _0^\# }\mathbb{G}_{m,Y}[1]
\right) \\
{}^{(\mathcal{P}\times _FY)\circ \pi ^\# }\uparrow&
&\uparrow&\\
P_Y&\xrightarrow{(\ref{Pic-sch})}&R^1\pi _* \mathbb{G}_{m,A\times Y}&
\end{array} \]
\end{claim}
%CCCCCCCCCCCCCCC

If we prove Claim \ref{mou-chotto},
the upper path from $P_Y$ to $\mathbb{G}_{m,Y}[1]$
is equal to $b_{\mathcal{P}}(f)$ and the lower path
is equal to $b_*(f)$, so Propositon \ref{albdefA}
follows.

Claim \ref{mou-chotto} follows if the map
\[ P_Y(Y)\to H^0(Y,R\pi _*\mathbb{G}_{m,A\times Y}[1])
=(R^1\pi _* \mathbb{G}_{m,A\times Y})(Y)=\mathrm{Pic}(A\times Y,0\times Y)\]
obtained by applying $H^0(Y,-)$ to
$(\mathcal{P}\times _FY)\circ \pi ^\#$
is the same as (\ref{Pic-sch}) for any $F$-scheme $Y$.

For this it suffices to show that the map 
obtained by applying $H^0(Y,-)$ to
$R\pi _*(\mathcal{P}\times _FY)$ in (\ref{P_f})
%%%%%%%%%%%%%%%%%%%%%%%%%%%%%%
\[ \begin{array}{rccccccc}%%%%
(\pi _*P_{A\times Y})(Y)
&\to R^1\pi _*\mathbb{G}_{m,A\times Y}(Y)=\mathrm{Pic}(A\times Y,0\times Y)
\end{array}\]%%%%%%%%%%%%
%%%%%%%%%%%%%%%%%%%%%%%%%
is equal to the canonical inclusion (\ref{Pic-sch})
\[ (\pi _*P_{A\times Y})(Y)\subset \mathrm{Pic}(A\times (A\times Y)
,0\times (A\times Y)) \] 
followed by
pull-back by $(\mathrm{diag}_A)\times \mathrm{id}_Y:A\times Y\to A\times A\times Y$. 

Replacing $A\times Y$ by an arbitrary $F$-scheme $Y$,
we are reduced to:

%LLLLLLLLLLLLLLLLLLLLLLL
\begin{lemma}
For any $A$-scheme $f:Y\to A$, the following commutes:
\[ \begin{array}{cccccc}
P_Y(Y)&\xrightarrow{\mathcal{P}\times _AY}&H^1(Y,\mathbb{G}_m)\\ {}\\
\cap ^{(\ref{Pic-sch})}&\nearrow _{\Gamma _f^*}& \\
\mathrm{Pic}(A\times Y,0\times Y)
\end{array}\]
\end{lemma}
%LLLLLLLLLLLLLLLLLLLLLL
But this is clear because by {\S}\ref{Pic-compat} and 
(\ref{Pic-sch})
both maps send $g:Y\to P$ to the pull-back of
the $\mathbb{G}_m$-torsor $\mathcal{P}$ on $A\times P$ to $Y$
by the morphism $(f,g):Y\to A\times P$.
This completes the proof of Proposition \ref{albdefA}.

%RRRRRRRRRRRRRRRRRRRR
\begin{remark}
Proposition \ref{albdefA} answers the question
raised by van Hamel in \cite[Remark 3.4]{hamel}.
\end{remark}
%RRRRRRRRRRRRRRRRRRRR

\section{Another surjectivity result}
%%%%%%%%%%
%%%%%%%%%%
\begin{theorem}\label{th2}
Let $K$ be a Henselian discrete valuation field of characteristic $(0, p)$ ($p$ positive) 
with residue field $k$ over which any principal homogeneous space under 
any abelian variety is trivial (e.g.\ a finite field or a separably closed field).
Assume that $v_K(p)< p-1$ where $v_K$ is the normalized additive valuation of $K$.

Let $X$ be a smooth projective variety over $K$ with good reduction. Assume $X$ has a degree $1$ zero-cycle (always true if $k$ is finite or separably closed).

Then, the albanese map
\[ \mathrm{CH}_0(X)^0 \to \mathrm{Alb}_X(K) \]
is surjective.
\end{theorem}
%%%%%%%%%%
%%%%%%%%%%
\begin{proof}
Since $X$ has a degree 1 zero-cycle, we may assume $X$ has a rational point
by a trace argument.

By a Bertini theorem over discrete valuation rings due to U. Jannsen and S. Saito
\cite[Theorem 4.2]{ss2}, 
there is a smooth curve $C \subset X$ which is obtained by repeated hypersurface sections 
and contains a rational point and has good reduction.

As is explained in \cite[Proposition 2.4]{spacefilling} 
such a hypersurface section has the property that the homomorphism
\[ \mathrm{Alb}_C \to \mathrm{Alb}_X \]
is smooth and has a connected kernel $N$.

Now $\mathrm{Alb}_C (=J_C)$ has good reduction because $C$ has good reduction.
We have now an exact sequence of Abelian varieties
\[ 0 \to N \to J_C \to \mathrm{Alb}_X \to 0. \]
By \cite[Theorem 4 on p.187]{BLR}, which is applicable by the assumption $v_K(p)< p-1$,
the induced sequence of N\'eron models
\[ 0 \to \mathcal{N} \to \mathcal{J}_C \to \mathcal{A}\mathit{lb}_X \to 0 \]
is exact and $\mathcal{N}$ and $\mathcal{A}\mathit{lb}_X$ also have good reduction. 
Therefore it induces an exact sequence of Abelian varieties over $k$ 
\[ 0 \to \overline{N} \to \overline{J_C} \to \overline{\mathrm{Alb}_X} \to 0. \]
(Here overline means reduction, not the algebraic closure.)
~\newline

Now we are going to establish the surjectivity of $J_C(K) \to \mathrm{Alb}_X(K)$; 
if it is proved our assertion follows from the 
known fact that $\mathrm{CH}_0(C)^0\to J_C(K)$ is an isomorphism
as $C$ has a rational point.
Pick any $a \in \mathrm{Alb}_X(K)$. We show $N_a=J_C \times _{\mathrm{Alb}_X} a$ 
has a rational point. 
The section $a$ naturally extends to a section 
$a' \in \mathcal{A}\mathit{lb}_X (\mathcal{O}_k)$  
and induces a section 
$\overline{a}\in \overline{\mathrm{Alb}_X}(k)$. 
We consider the scheme $\mathcal{N}_{a'}=\mathcal{J}_C \times _{\mathcal{A}\mathrm{lb}_X} a'$.
Its special fiber $\overline{N}_{\overline a}
=\overline{J_C}\times _{\overline{\mathrm{Alb}_X}} \overline{a}$ is a torsor over the field $k$ 
under the abelian variety $\overline{N}$. 
By the assumption on the residue field, it has a rational point. 
By Hensel's lemma, the rational point lifts to a section of $\mathcal{N}_{a'}$, 
giving a rational point on $N_a$.
\end{proof}

\begin{acknowledgements}
The author thanks his advisor Professor Shuji Saito for always encouraging and guiding him. 
A part of this work was done while the author was staying at the Universit\"at Duisburg-Essen in 2014.
He thanks it for its hospitality.
The referee has helped the author improve the exposition and pointed out some errors in an earlier draft.
%This work was supported by the Program for Leading Graduate 
%Schools, MEXT, Japan.
\end{acknowledgements}

\end{document}